\documentclass[11pt]{article}
\usepackage{amsmath,amsxtra,amssymb,latexsym,
	epsfig,amscd,amsthm,fancybox,epsfig}
\usepackage[mathscr]{eucal}
\usepackage{graphicx}
\usepackage{epsfig} 
\usepackage{epstopdf}
\usepackage{cases}
\usepackage{subfig}
\usepackage{color}
\usepackage{hyperref}

\newtheorem{thm}{Theorem}[section]

\newtheorem{deff}{Definition}[section]

\newtheorem{lem}{Lemma}[section]
\newtheorem{prop}{Proposition}[section]
\theoremstyle{definition}

\theoremstyle{remark}
\newtheorem{rem}{Remark}

\numberwithin{equation}{section}

\newcommand{\eps}{\varepsilon}

\newcommand{\F}{\mathcal{F}}

\newcommand{\E}{\mathbb{E}}

\newcommand{\PP}{\mathbb{P}}

\newcommand{\R}{\mathbb{R}}

\numberwithin{equation}{section}

\newcommand{\bed}{\begin{displaymath}}
\newcommand{\eed}{\end{displaymath}}
\newcommand{\bea}{\bed\begin{array}{rl}}
	\newcommand{\eea}{\end{array}\eed}

\newcommand{\barray}{\begin{array}{ll}}
	\newcommand{\earray}{\end{array}}

\newcommand{\1}{\boldsymbol{1}}

\def\bar{\overline}
\def\hat{\widehat}
\def\a.s{\text{\;a.s.\;}}

\begin{document}
	\title{Exponential tightness of 
a family of Skorohod integrals\thanks{This research
			was supported in part by the National Science Foundation under grant
			DMS-1710827.}}
	\author{Nhu N. Nguyen\thanks{Department of Mathematics, University of Connecticut, Storrs, CT
			06269, USA, nguyen.nhu@uconn.edu}
		}
	\maketitle

\begin{abstract}
Exponential tightness of a family of Skorohod integrals is studied in this paper. We first provide a counterexample to illustrate that in general the exponential tightness with speed $\eps$ similar to It\^o integral does not hold, even for any speed $\eps^\alpha$ with $\alpha>0$. Then, some characterizations of this subject are given. Application is also provided to illustrate the proposed results.

\medskip

\noindent {\bf Keywords.}
Exponential tightness; Malliavin calculus; Anticipating integrals

\medskip
\noindent{\bf Subject Classification.}
60H05; 60H07

\medskip
\noindent{\bf Running Title.} Exponential tightness of Skorohod integrals
\end{abstract}

\section{Introduction}

Large deviations principle (see e.g., \cite{DZ92} for a full 
discussion
of this subject) plays an important role in both theory and application such as averaging principle
of fast-slow systems, equilibrium and non-equilibrium statistical mechanics, multi-fractals, and
thermodynamic formulation of chaotic systems; see \cite{DZ92,Tou09} and the reference therein.
In Polish space, the exponential tightness is a necessary condition for the large deviations principle (with inf-compact rate function) and implies the large deviations relative compactness (i.e., every sub-sequence contains another sub-sequence satisfying the large deviations principle with some rate function); see e.g., \cite{DZ92}.
Hence, this property has a crucial role in the large deviations theory, for example, see
\cite{Gui03,Lip96,NY20,Puh16}.
Moreover, since it provides a kind of ``exponential tail estimates for the tightness", it is also very interesting in its own right in stochastic analysis.
Recall that we say a family of random variables $\{F_\eps\}_{\eps>0}$ in $\R^d$ is exponentially tight with
speed $v(\eps)$ satisfying $v(\eps)\to0$ as $\eps\to0$, if
$$
\lim_{L\to\infty}\limsup_{\eps\to0} v(\eps)\log\PP(|F_\eps|>L)=-\infty.
$$
It is noted that if $v_1(\eps)<v_2(\eps)$ (as $\eps\to0$) then the exponential tightness with the speed $v_1(\eps)$ implies the exponential tightness with the speed $v_2(\eps)$.

Given a $d$-dimensional standard Brownian motion $W(t)$ defined on the
canonical probability space $(\Omega,\F,\{\F_t\},\PP)$ and the Hilbert space $H=L^2([0,1],\R^d)$, consider an isonormal Gaussian process $W=\{W(h):h\in H\}$ defined by $W(h)=\int_0^1 h(t)dW(t)$.
Denote by $D$ the Malliavin derivative operator and by $\delta$ its adjoint, the divergence operator, which is called
 Skorohod integral in our setting; see \cite{Nua06}.
Let $\{u_\eps\}_{\eps>0}$ be a family of Skorohod integrable processes  and we consider the family of random variables
$$
F_\eps=\sqrt\eps\delta(u_\eps).
$$
As a special case, it is well-known in the classical It\^o stochastic calculus that if $u_\eps$ is a non-anticipating process, $\delta(u_\eps)$ can be understood in the It\^o sense.
If we assume further that $u_\eps$ are bounded (uniformly in $\eps$) by $K$, i.e., $|u_\eps(t)|\leq K,\forall t\in[0,1],\eps>0$ a.s., then $\{F_\eps\}_{\eps>0}$ is exponentially tight with the speed $v(\eps)=\eps$. To be more precise, by Schilder's theorem \cite[Lemma 5.2.2]{DZ92} one has
$$
\PP(|F_\eps|>L)\leq 4d\exp\Big\{-\frac{L^2}{2d\eps K^2}\Big\}.
$$
What about the exponential tightness of the family $\{F_\eps\}_{\eps>0}$ in the general case?
This work aims to address such a question.

First, we are wondering
 if one relaxes the measurability and allows $\{u_\eps\}_{\eps>0}$ to be anticipating processes and keeps the (uniform) boundness,
 is the family $\{F_\eps\}_{\eps>0}$ still exponentially tight with the speed $v(\eps)=\eps$ or at least with some speed $v(\eps)=\eps^\alpha$ (for some $\alpha\in(0,1)$)? Unfortunately, it is not true in general. We first provide a counterexample as follows.
\begin{thm}\label{thm-1}
Assume $W(t)$ is a one-dimensional Brownian motion.
	Let $f:\R\to\R$ be the
	function
	defined by
	$
	f(x)=(0\vee x\wedge 1)^{3/4}. 
	$
	Define
	$X=f(W(1))$
	and
	$$F_\eps=\sqrt\eps\delta\Big(X\1_{[0,1]}(t)\Big).$$
	The family $\{F_\eps\}_{\eps>0}$ is not exponentially tight with any speed $v(\eps)=\eps^\alpha$, $\alpha>0$.
	Moreover, for any $\alpha>0$ one has
	\begin{equation}\label{eq-thm1}
	\limsup_{L\to\infty}\limsup_{\eps\to0}\eps^\alpha\log\PP(|F_\eps|>L)=0.	
	\end{equation}
\end{thm}
Therefore, it is very natural to ask the question: under
what
conditions on $u_\eps$, is the family $\{F_\eps\}_{\eps>0}$ 
exponentially tight (with some suitable speed)?
We aim to provide some sufficient conditions for the exponential tightness of this family of Skorohod integrals. 

It will be seen from the above counterexample that the family $\{F_\eps\}_{\eps>0}$ may not be exponentially tight because the relationship 
between
the integrand and the whole path of Brownian motion
is somewhat uncontrollable. This relation is often
described
by the derivative of the integrand with respect to the Brownian motion. Therefore, if we can control the moments of the derivative of the integrand, we can have the exponential tightness with some suitable speed. To be precise, we have the following theorem.
[In this paper, $\otimes$ denotes the tensor product; and for a normed space $V$ endowed with the norm $\|\cdot\|_V$, and $V$-valued random variable $v$,  $\|v\|_{L^p(\Omega,V)}:=\big(\E(\|v\|_V^p)\big)^{1/p}$.]

\begin{thm}\label{thm-2}
	Let $\{u_\eps\}_{\eps>0}$ be a family of Skorohod integrable stochastic processes.
	Assume that there are $\kappa_1>-1/2$ and $\kappa_2\geq 0$ such that
	\begin{equation}\label{eq-asp-2}
 \|u_\eps\|_{L^p(\Omega,H)}+\|Du_\eps\|_{L^p(\Omega,H\otimes H)}\leq c\eps^{\kappa_1}p^{\kappa_2},\forall \eps>0, p\geq 1,
	\end{equation}
	for some universal constant $c$, independent of $\eps,p$.
	Then the family $\{F_\eps=\sqrt\eps\delta(u_\eps)\}_{\eps>0}$ is exponentially tight with the speed $v(\eps)=\eps^\alpha$ for any $\alpha$ satisfying
	\begin{equation}\label{eq-rate}
	\alpha<\frac{0.5+\kappa_1}{5+\kappa_2}.
	\end{equation}
\end{thm}
This theorem is proved by using Meyer's inequality after revisiting its proof with the use of best constants and the estimate of the exponential moment.
As seen in the above theorem, compared with the speed in the It\^o case, we have to pay the cost in the speed for the non-adaptedness.

If we can control the moments of the second-order Malliavin derivative (denoted by $D^2$) of the integrands, we can combine It\^o's formula for Skorohod integral and Meyer's inequality to estimate the $p$th-moment and 
obtain the following result, which can be used to improve the speed in certain cases.

\begin{thm}\label{thm-3}
	Let $\{u_\eps\}_{\eps>0}$ be a family of Skorohod integrable stochastic processes.
	Assume that there are $\bar\kappa_1,\bar\kappa_3>-1/2$, and $\bar \kappa_2,\bar \kappa_4\geq 0$ such that
	\begin{equation}\label{eq-asp-3}
	\|Du_\eps\|_{L^p(\Omega\times H^{\otimes 2})}+\|D^2u_\eps\|_{L^p(\Omega\times H^{\otimes 3})}\leq c\eps^{\bar\kappa_1}p^{\bar\kappa_2},\quad
	\|u_\eps\|_{L^p(\Omega);H} \leq c'\eps^{\bar\kappa_3}p^{\bar\kappa_4},\;\forall \eps>0, p\geq 1,
	\end{equation}
	for some universal constants $c,c'$, independent of $\eps,p$.
	In the above,
	$\|Du_\eps\|_{L^p(\Omega\times H^{\otimes 2})}:=\big(\int_0^1\int_0^1\E| D_su_\eps(r)|^{p}drds\big)^{1/p},$
	$\|D^2u_\eps\|_{L^p(\Omega\times H^{\otimes 3})}:=\big(\int_0^1\int_0^1\int_0^1 \E| D_tD_su_\eps(r)|^{p}drdtds\big)^{1/p},$	
	and
	$
		\|u_\eps\|_{L^p(\Omega);H}:=\big(\int_0^1 (\E |u_\eps(s)|^p)^{2/p}ds\big)^{1/2}.
	$
	Then the family $\{F_\eps=\sqrt\eps\delta(u_\eps)\}_{\eps>0}$ is exponentially tight with the speed $v(\eps)=\eps^\alpha$ for any $\alpha$ satisfying
	$$
	\alpha<\frac{0.5+\hat\kappa_1}{\hat\kappa_2},
	$$
	where $2\hat\kappa_1=\min\{\bar\kappa_1+\bar\kappa_3,2\bar\kappa_3\}$, $2\hat\kappa_2=\max\{6+\bar\kappa_2+\bar\kappa_4,1+2\bar\kappa_4\}$.
\end{thm}

The rest of the paper is organized as follows. Section \ref{sec:pre} recalls briefly the Malliavin calculus emphasizing the Malliavin derivative operator, its adjoint operator, and
Meyer's inequality.
The proofs of our main results are given in
Section \ref{sec:prof}. An application to mathematical physics is given in Section \ref{sec:app}.

\section{Malliavin Calculus}\label{sec:pre}
\subsection{Malliavin derivative and Skorohod integral}
We recall briefly some basic definitions in Malliavin calculus and refer the readers to \cite{Nua06} for a full construction.
Let $W(t)$ be a $d$-dimensional standard Brownian motion defined on the
canonical probability space $(\Omega,\F,\{\F_t\},\PP)$ and $H=L^2([0,1],\R^d)$. Consider $W =\{W(h), h \in H\}$, the space of Gaussian isonormal processes defined by $W(h)=\int_0^1 h(t)dW(t)$.
Denote by
$$
\mathcal S = \Big\{F = f(W(h_1), \dots, W(h_n))| f\in\mathcal C^\infty_p(\R^n), h_i \in H, n \geq 1 \Big\},
$$
the class of smooth random variables, where $\mathcal C^\infty_p(\R^n)$ is the space of smooth function $f\in\mathcal C
^\infty$ with polynomial
growth on derivatives.
\begin{deff}$($See \cite[Definition 1.2.1]{Nua06}$)$
	The derivative of a smooth random variable $F\in\mathcal S$ is the $H$-valued random variable given by
	$$
	DF =\sum_{i=1}^n
	\frac{\partial f}{\partial x_i}(W(h_1),\dots , W(h_n))h_i.$$
\end{deff}
Define the norm
$$
\|F\|_{1,2}^2=\E\Big(|F|^2+\|DF\|_H^2\Big).
$$
Let $\mathbb D^{1,2}$ be the closure of $\mathcal S$ with respect to the norm $\|\cdot\|_{1,2}$.
 One can extend $D$ on $\mathcal S$ as a closed operator on $\mathbb D^{1,2}$. We call this operator $D$ the Malliavin derivative operator.
Moreover,
we can also
define the iteration of the operator $D$ in such a way that for a smooth random variable $F$, the iterated derivative $D^kF$ is a random
variable with values in $H^{\otimes k}$, then we denote by $\mathbb D^{k,p}$ the completion of the family of smooth random variables $\mathcal S$ with respect to the norm $\|\cdot\|_{k,p}$ defined by
$$
\|F\|_{k,p} =
\Bigg( \E(|F|^p) +\sum_{j=1}^k\E(\|D^jF\|^p_{H^{\otimes j}})\Bigg)^{\frac 1p}.
$$
For $k = 0$, we
use the convention $\|\cdot\|_{0,p} = \|\cdot\|_p$
and $\mathbb D^{0,p} = L^p(\Omega)$. 
\begin{deff}$($See \cite[Definition 1.3.1]{Nua06}$)$
	Denote by $\delta$ the adjoint of the operator D, i.e.,
	$\delta$ is an unbounded operator on $L^2(\Omega; H)$ with values in $L^2(\Omega)$ such that:
	\begin{itemize}
	\item[]{\rm (i)} The domain of $\delta$, denoted by {\rm Dom}$\delta$, is the set of $H$-valued square
	integrable random variables $u\in L^2(\Omega; H)$ such that
	$$|\E(\langle DF, u\rangle_H)| \leq c\|F\|_{2},$$
	for all $F\in\mathbb D^{1,2}$, where c is some constant depending on $u$.
	\item[]{\rm (ii)} If u belongs to {\rm Dom}$\delta$, then $\delta(u)$ is the element of $L^2(\Omega)$ characterized by the following expression
	$$\E(F\delta(u)) = \E(\langle DF,u\rangle_H) \text{ for any }F \in\mathbb D^{1,2}.
	$$
	\end{itemize}
\end{deff}
The operator $\delta$ is called the divergence operator and is closed
since $D$ is an unbounded and densely defined operator.
In our case, $H$ is
an $L^2$ space, the elements of Dom$\delta$ are square integrable processes, and $\delta(u)$ is called the Skorohod stochastic integral.
One says $u$ is Skorohod integrable if $\delta(u)$ is well-defined. We will often use the notation
$$
\int_0^1 u(s)\delta W(s):=\delta(u)\text{ and }\int_0^tu(s)\delta W(s):=\delta(u\1_{[0,t]}).
$$

\subsection{Meyer's inequality and the best constants}
Meyer's inequality gives us an effective way to bound the moments of the Skorohod integral by the moments of the integrand and its derivative.
Denote by $\mathbb L^{1,2}$ the class of
processes $u\in L^2([0,1]\times\Omega)$ such that $u(t)\in \mathbb D^{1,2}$ for almost all $t$, and there
exists a measurable version of the two-parameter process $D_su(t)$ verifying
$\E \int_0^1\int_0^1(D_su(t))^2dsdt <\infty$.
[Here, it is noted that for each $t$, $u(t)$ is a random variable; and $Du(t)$ is a $H$-valued random variable and thus, can be parameterized as $D_su(t)$.]

We first recall
Meyer's inequality; see e.g., \cite[Theorem 1.5.1]{Nua06}.

\begin{prop}\label{prop-meyer}
	For $1 < p <\infty$ there are constants $\bar K_p$ such that
	\begin{equation}\label{eq-prop-0}
	\|DF\|_{L^p(\Omega,H\otimes H)}\leq \bar K_p\|CF\|_p, 
	\end{equation}
	for any random variable $F\in \mathbb D^{1,p}$, where $C:=-\sqrt{-L}$ and $L=-\delta D$ is the Ornstein-Uhlenbeck operator $($see \cite[Section 1.4]{Nua06} for the definition$)$.
\end{prop}
Therefore, 
we can obtain from Proposition \ref{prop-meyer} the following estimate; see e.g., \cite[Proposition 1.5.4]{Nua06}.

\begin{prop}\label{prop-meyer2}
	Let $u$ be a stochastic	 process in $\mathbb L^{1,2}$, and let $p > 1$. Then we have
	\begin{equation}\label{eq-prop-1}
	\Big(\E |\delta(u)|^p\Big)^{\frac 1p} \leq
	K_p\Bigg(\|u\|_{L^p(\Omega,H)}+\|Du\|_{L^p(\Omega,H\otimes H)}\Bigg).
	\end{equation}
\end{prop}

To estimate the speed of the exponential tightness, we want to use the best constants, and have
 precise estimates for the constants in the estimate \eqref{eq-prop-1}. We have the following result.
\begin{prop}\label{prop-meyer3}
	The constants $K_p$, $p\geq 2$ in \eqref{eq-prop-1} can be 
		can be
		chosen smaller than a positive constant multiple of $p^{5}$, i.e., 
		there is a universal constant $c$ such that $K_p\leq cp^5$ for all $p\geq 2$.
\end{prop}
\begin{proof}
	We prove this proposition by revisiting the proof of Proposition 1.5.4 in \cite{Nua06} with the use of the best constants for Meyer's inequality and taking care of the norms of multiplier operators using in the arguments. 
	In the below, we use
	the letter $c$ to represent universal constants (independent of $\eps,p$), whose values may
	change for different usage.
	
	We first revisit the generalized Meyer's inequality \cite[Theorem 1.5.1]{Nua06} and have the following lemma.
	\begin{lem}\label{lem-add} For $1<q<\infty$, there are constants $\bar K_{2,q}$ satisfying that for all polynomial random variable $G$,
		\begin{equation}\label{eq-524}
		\|D^2G\|_{L^q(\Omega,H\otimes H)}\leq \bar K_{2,q}\|C^2G\|_{q},
		\end{equation}
		and that as $q\to 1$, $\bar K_{2,q}\leq \frac{c}{(q-1)^3}$. 
	\end{lem} 
    \begin{proof}
    	The proof of \eqref{eq-524} is given in \cite[Proof of Theorem 1.5.1, page 73]{Nua06}. Revisiting all the computations in \cite[page 73]{Nua06}, we obtain that $\bar K_{2,q}$ can be chosen such that
    	$$
    	\bar K_{2,q}\leq cA_q^{1/q}\bar K_q^2 \|R\|_{L^q(\Omega)\to L^q(\Omega)},
    	$$
where $A_q$ is the $q$-th moment of a Gaussian variable (see \cite[Appendix A.1]{Nua06}) and
    	$
    	R=\sum_{n=1}^\infty \sqrt {1-\frac{1}{n}}J_n,
    	$
    	$J_n$ is the projection operator onto $n$-th Wiener chaos. The operator $R$ is used to exchange the derivative operator (by using the commutativity relationship \cite[Lemma 1.4.2]{Nua06}).
    	By applying the multiplier theorem \cite[Theorem 1.4.2]{Nua06} and \cite[Lemma 1.4.1]{Nua06}, one can obtain that $R$ is a bounded operator and as $q\to 1$, its norm (when regarded as an operator from $L^q(\Omega)$ to $L^q(\Omega)$) can be bounded by $\frac c{q-1}.$
    	Finally, it follows from \cite{Lar02} that as $q\to 1$, $\bar K_{q}\leq \frac c{q-1}$.
    	Therefore, we can complete the proof of this lemma.
    \end{proof}
Now, let $q$ be the conjugate of $p$ and $G$ be any
polynomial random variable with $\E(G) = 0$.
	Under the convention ($\E u=0$) for simplicity as in \cite[Proof of Proposition 1.5.4]{Nua06}, we have from \cite[Proof of Proposition 1.5.4]{Nua06}, Lemma \ref{lem-add}, the commutativity relationship \cite[Lemma 1.4.2]{Nua06} that
	\begin{equation}\label{eq-513}
	\begin{aligned}
	|\E (\delta(u)G)|\leq& \|Du\|_{L^p(\Omega;H\otimes H)}\|DC^{-2}DG\|_{L^q(\Omega,H\otimes H)}\\
	=& \|Du\|_{L^p(\Omega;H\otimes H)}\|D^2C^{-2}RG\|_{L^q(\Omega,H\otimes H)}\\
	\leq & \bar K_{2,q} \|Du\|_{L^p(\Omega;H\otimes H)}\|RG\|_{q}\\
	\leq &\bar K_{2,q} \|R\|_{L^q(\Omega)\to L^q(\Omega)}\|Du\|_{L^p(\Omega;H\otimes H)}\|G\|_{q}.
	\end{aligned}
	\end{equation}
	In the above,
	$
	R=\sum_{n=2}^\infty \frac{n}{n-1}J_n,
	$
	which is used to exchange the derivative operator by using the commutativity relationship.
By applying the multiplier theorem \cite[Theorem 1.4.2]{Nua06} and \cite[Lemma 1.4.1]{Nua06}, one can obtain that $R$ is a bounded operator and as $q\to 1$, its norm (when regarded as an operator from $L^q(\Omega)$ to $L^q(\Omega)$) can be bounded by $\frac c{(q-1)^2}=c(p-1)^2$; and thus, as $p\to\infty$, 
$
\|R\|_{L^q(\Omega)\to L^q(\Omega)}\leq cp^2.
$
Moreover, it follows from Lemma \ref{lem-add} that as $q\to 1$, $\bar K_{2,q}\leq \frac c{(q-1)^3}$; and thus, as $p\to\infty$, 
$
\bar K_{2,q}\leq cp^3.
$
Therefore, we obtain from \eqref{eq-513} that there is a universal constant $c$, which is independent of $p\geq 2$ such that
\begin{equation}\label{eq-513-1}
\begin{aligned}
|\E (\delta(u)G)|\leq cp^5\|Du\|_{L^p(\Omega;H\otimes H)}\|G\|_{q}.
\end{aligned}
\end{equation}
 The proposition follows from \eqref{eq-513-1} after taking the supremum with respect to polynomial random variable $G$ as in the standard duality argument \cite{Nua06}.
\end{proof}

\section{Proof of Main Results}\label{sec:prof}
\begin{proof}[Proof of Theorem \ref{thm-1}]
Direct calculation shows that
$$
f'(x)=
\begin{cases}
0\text{ if }x<0\text{ or }x>1,\\
\frac{3x^{-1/4}}4 \text{ if }0<x<1.
\end{cases}
$$
Therefore, $f'(W(1))\in L^2(\Omega)$.
Applying the chain rule (see e.g., \cite{Nua06}, \cite[Proposition 2.3.1]{M95} or \cite[Theorem 5.7]{Koc18}), we have that $DX=Df(W(1))=f'(W(1))\1_{[0,1]}(t)\in L^2(\Omega,L^2([0,1]))$.
Moreover, it is readily seen that $X\1_{[0,1]}(t)$ is Skorohod integrable.
 Using integration
 by parts (see \cite[Proposition 1.3.3]{Nua06}), we have that
	\begin{equation}\label{eq-prof1-0}
	\begin{aligned}
	F_\eps&=\sqrt\eps \delta\Big(X\1_{[0,1]}(t)\Big)\\
	&=\sqrt \eps X\delta(1_{[0,1]}(t))-\sqrt\eps\int_0^1 D_rXdr\\
	&=\sqrt\eps XW(1)-\sqrt\eps f'(W(1)).
	\end{aligned}
	\end{equation}
	
To prove Theorem \ref{thm-1}, it suffices to prove \eqref{eq-thm1} for any $\alpha\in(0,1)$.
	We aim to use a contradiction argument by assuming that 
	\begin{equation}\label{eq-prof1-1}
	\limsup_{L\to\infty}\limsup_{\eps\to0}\eps^\alpha\log\PP(|F_\eps|>L)\leq\eta\in(-\infty,0).	
	\end{equation}
	Since $X$ is bounded almost surely and $W(1)$ is Gaussian, we can obtain that the family $\big\{\sqrt\eps XW(1)\big\}_{\eps>0}$ is exponentially tight with speed $v(\eps)=\eps$. Thus, one has
	$$
	\limsup_{L\to\infty}\limsup_{\eps\to0}\eps^\alpha\log\PP\Big(\Big|\sqrt\eps XW(1)\Big|>L\Big)=-\infty.
	$$
	As a result, the family $\big\{F_\eps-\sqrt \eps XW(1)\big\}_{\eps>0}$ satisfies that
	\begin{equation}\label{eq-prof1-2}
	\begin{aligned}
	\limsup_{L\to\infty}&\limsup_{\eps\to0}\eps^\alpha\log\PP\Big(\Big|F_\eps-\sqrt\eps XW(1)\Big|>L\Big)\\
	&\leq\limsup_{L\to\infty}\limsup_{\eps\to0}\eps^\alpha\log\bigg\{\PP\Big(|F_\eps|>\frac L2\Big)+\PP\Big(\Big|\sqrt\eps XW(1)\Big|>\frac L2\Big)\bigg\}\\
	&=\max\bigg\{\limsup_{L\to\infty}\limsup_{\eps\to0}\eps^\alpha\log\PP\Big(|F_\eps|>\frac L2\Big),
	\limsup_{L\to\infty}\limsup_{\eps\to0}\eps^\alpha\log\PP\Big(\Big|\sqrt\eps XW(1)\Big|>\frac L2\Big)\bigg\}\\
	&\leq\eta<0.
	\end{aligned}
	\end{equation}
	A consequence of \eqref{eq-prof1-0} and \eqref{eq-prof1-2} is that
	$$
		\limsup_{L\to\infty}\limsup_{\eps\to0}\eps^\alpha\log\PP\Big(\Big|\sqrt\eps f'(W(1))\Big|>L\Big)\leq\eta,
	$$
and then one gets
\begin{equation}\label{eq-prof1-3}
\limsup_{L\to\infty}\limsup_{\eps\to0}\eps^\alpha\log\PP\Big(\Big| f'(W(1))\Big|>L\eps^{-1/2}\Big)\leq\eta.
\end{equation}

Now, denote
$Y= f'(W(1)).$
From \eqref{eq-prof1-3}, there are $L_0=L_0(\eta)>1$ and $\eps_0=\eps_0(L_0,\eta)<1$ such that for all $L>L_0$, $ \eps<\eps_0$
$$
\PP(|Y|>L\eps^{-1/2})\leq \exp\big\{\frac{\eta{\color{blue}\eps^{-\alpha}}}2\big\}.
$$
Particularly, let $\eps=\frac {\eps_0L_0}{L^2}$, one has for all $L>L_0$
$$
\PP\Big(|Y|>L^2(\eps_0L_0)^{-1/2}\Big)\leq \exp\Big\{\frac{\eta(\eps_0L_0)^{-\alpha} L^{2\alpha}}2\Big\},
$$
which implies that
\begin{equation}\label{eq-prof1-4}
\PP\Big(|Y|^{\alpha/2}>L^\alpha(\eps_0L_0)^{-\alpha/4}\Big)\leq \exp\Big\{\frac{\eta(\eps_0L_0)^{-\alpha} L^{2\alpha}}2\Big\},\; \forall L>L_0.
\end{equation}
We obtain from \eqref{eq-prof1-4} that for all $t>L_0^{3\alpha/4}\eps_0^{-\alpha/4}$
\begin{equation}\label{eq-prof1-5}
\PP\Big(|Y|^{\alpha/2}>t\Big)\leq \exp\{-c_0t^{2}\},
\end{equation}
where $c_0=\frac{-\eta(\eps_0L_0)^{-\alpha/2}}2>0$.
From \eqref{eq-prof1-5}, which is a kind of tail estimates of a sub-Gaussian random variable, we can get a kind of ``moments control property" for $|Y|^{\alpha/2}$ (see e.g., \cite[Lemma 1.4]{RH19}), i.e., for all $p>0$
\begin{equation}\label{eq-prof1-6}
\E|Y|^{\alpha p/2}\leq c_1c_2^p p^{p/2},
\end{equation}
for some constants $c_1,c_2$ depending only on $\eta,\eps_0,L_0,\alpha$ and being independent of $p$.

On the other hand, we have
\begin{equation}\label{eq-1111}
\begin{aligned}
\E|Y|^p=&\int_{-\infty}^\infty |f'(x)|^p\frac{e^{-\frac {x^2}2}}{\sqrt{2\pi}}dx\geq \int_{0}^1 |f'(x)|^p\frac{e^{-\frac {x^2}2}}{\sqrt{2\pi}}dx\\
\geq& \frac{0.75^p}{\sqrt{2e\pi}}\int_0^1 x^{-0.25p}dx\\
=&\infty \text{ if }p\geq 4.
\end{aligned}
\end{equation}
Combining \eqref{eq-prof1-6} and \eqref{eq-1111} leads to a contradiction. So, we obtain \eqref{eq-thm1} and complete the proof.

\end{proof}

\begin{rem}
	As was seen in the above counterexample, although the process $u_\eps(t)=f(W(1))\1_{[0,1]}(t)$ is bounded uniformly, the corresponding family of Skorohod integrals is not exponentially tight with any speed $\eps^\alpha$, $\alpha>0$. The reason is that the relationship between the non-adapted intergrand and the whole paths of the Brownian motion is ``uncontrollable", which is illustrated by \eqref{eq-1111}. Since all moments of the Malliavin derivative do not exist, the assumptions in Theorems \ref{thm-2} and \ref{thm-3} are violated.
\end{rem}

\begin{proof}[Proof of Theorem \ref{thm-2}]
	Denote
	$
	\alpha=(0.5+\kappa_1)\beta
	$
	for some $\beta$ satisfying
	$
	\beta<\frac 1{5+\kappa_2}<1.
	$
	We have from Meyer's inequality with precise constants (Proposition \ref{prop-meyer2} and \ref{prop-meyer3}) and \eqref{eq-asp-2} that
	$$
	\begin{aligned}
		\E \exp\{|\eps^{-\kappa_1}\delta(u_\eps)|^{\beta}\}&=\sum_{n=0}^\infty \frac{\E |\eps^{-\kappa_1}\delta(u_\eps)|^{n\beta}}{n!}
		\leq \sum_{n=0}^\infty \frac{\big(\E |\eps^{-\kappa_1}\delta(u_\eps)|^{n}\big)^{\beta}}{n!}\\
		&\leq c_3+\sum_{n=2}^\infty \frac{\eps^{-n\kappa_1\beta}K_n^{n\beta}\|u_\eps\|_{1,n}^{n\beta}}{n!}		\\
		&\leq c_3+c_4\Big(\sum_{n=0}^\infty  \frac{n^{(5+\kappa_2)n\beta}}{n!}\Big)\leq c_5,
	\end{aligned}
	$$
	for some constants $c_3,c_4, c_5$, independent of $n,\eps$.
	In the above,
the last estimate follows from the fact that $\sum_{n=0}^\infty \frac{n^{(5+\kappa_2)n\beta}}{n!}<\infty$, which is implied
by  the fact that $(5+\kappa_2)\beta<1$ and the ratio test.

By Markov's inequality, one has that for any $L>0$,
$$
\begin{aligned}
\limsup_{\eps\to0}\;&\eps^\alpha\log \PP(|F_\eps|>L)
=\limsup_{\eps\to0}\eps^\alpha\log \PP(\eps^{0.5+\kappa_1}|\eps^{-\kappa_1}\delta(u_\eps)|>L)\\
&=\limsup_{\eps\to0}\eps^\alpha\log \PP(\eps^{(0.5+\kappa_1)\beta}|\eps^{-\kappa_1}\delta(u_\eps)|^\beta>L^\beta)\\
&=\limsup_{\eps\to0}\eps^\alpha\log \PP\Big(\exp\{|\eps^{-\kappa_1}\delta(u_\eps)|^\beta\}>\exp\{L^\beta \eps^{-(0.5+\kappa_1)\beta}\}\Big)\\
&\leq \limsup_{\eps\to0}\eps^\alpha\log\frac{\E\exp\{|\eps^{-\kappa_1}\delta(u_\eps)|^\beta\}}{\exp\{L^\beta\eps^{-(0.5+\kappa_1)\beta}\}}\\
&\leq \limsup_{\eps\to0}\eps^\alpha\log\frac{c_5}{\exp\{L^\beta\eps^{-(0.5+\kappa_1)\beta}\}}\\
&=-\limsup_{\eps\to0}\eps^{\alpha-(0.5+\kappa_1)\beta}L^\beta\\
&=-L^\beta.
\end{aligned}
$$
Therefore, the exponential tightness with the speed $v(\eps)=\eps^\alpha$ follows immediately.
\end{proof}
\begin{rem}
	Actually, the constant $5+\kappa_2$ in \eqref{eq-rate} comes from the order needed to control the $p$-th moment of $\delta(u_\eps)$.
	Moreover,
	let us come back to the non-anticipating stochastic integral case and let $u_\eps$ be constant for simplicity. In that case $\kappa_1=\kappa_2=0$ and $\delta(u_\eps)$ is Gaussian. It is well-known that the $p$-th moment of a Gaussian random variable is controlled by $p^{0.5p}$ only and thus, replacing $5$ in the denominator of the right hand side in \eqref{eq-rate} by $0.5$ somehow will bring us back to the results in the classical case (the case of It\^o integrals) as given by Schilder's theorem \cite[Lemma 5.2.2]{DZ92}.
\end{rem}
\begin{proof}[Proof of Theorem \ref{thm-3}]
	Let $Z(t)=\int_0^t u_\eps(t)\delta W(t):=\delta(u_\eps\1_{[0,t]})$.
	For the simplicity of notation, let us assume $W(t)$ and $u_\eps(t)$ have real values, i.e., the dimension $d=1$. [The general case ($d>1$) is the same by understanding appropriate calculations in their corresponding vector operations.]
		
	By It\^o's formula for Skorohod integral \cite[Theorem 3.2.2]{Nua06}, we have for $n\geq 2$
	\begin{equation}\label{eq-prof3-1}
	\begin{aligned}
	\big|Z(t)\big|^n=&\int_0^t n\big|Z(s)\big|^{n-1}u_\eps(s)\delta W(s)\\
	&+\int_0^t n(n-1)\big|Z(s)\big|^{n-2} \Big(\frac{|u_\eps(s)|^2}2+u_\eps(s)\int_0^s D_su_\eps(r)\delta W(r)\Big)ds.
	\end{aligned}
	\end{equation}
	Therefore, it follows from \eqref{eq-prof3-1} and H\"older's inequality that for $n>2$
	\begin{equation}\label{eq-prof3-2}
	\begin{aligned}
	\E \big|Z(t)\big|^n=&
	n(n-1)\E\int_0^t \big|Z(s)\big|^{n-2} \Big(\frac{u_\eps^2(s)}2+u_\eps(s)\int_0^s D_su_\eps(r)\delta W(r)\Big)ds\\
	\leq& n(n-1)\int_0^t \Big(\E\big|Z(s)\big|^n\Big)^{\frac{n-2}{n}}\Bigg(\Big(\E|u_\eps(s)|^n\Big)^{\frac 2n}+2\bigg( \E\Big|u_\eps(s)\int_0^s D_su_\eps(r)\delta W(r)\Big|^{n/2}\bigg)^{\frac 2n}\Bigg)ds.
	\end{aligned}
	\end{equation}
	It is known from Bihari-LaSalle inequality \cite{Las49} that if
	$$
	v(t)\leq a\int_0^tk(s)(v(s))^{\frac{n-2}{n}}ds,\quad\forall t\in[0,1],
	$$
	then we have
	$$
	v(t)\leq \Bigg(\frac{2a\int_0^t k(s)ds}{n}\Bigg)^{n/2},\quad\forall t\in[0,1].
	$$
	Applying this fact and \eqref{eq-prof3-2}, we deduce that
	\begin{equation}\label{eq-prof3-3}
	\begin{aligned}
	\E|Z(1)|^n\leq& \Bigg((n-1)\int_0^1 \Big(\E|u_\eps(s)|^n\Big)^{2/n}ds+2(n-1)\int_0^1\Big( \E\Big|u_\eps(s)\int_0^s D_su_\eps(r)\delta W(r)\Big|^{n/2}\Big)^{2/n}ds \Bigg)^{n/2}\\
	\leq& \Bigg(n\|u_\eps\|_{L^n(\Omega);H}^2
	+2n\int_0^1 \Big(\E\Big|u_\eps(s)\int_0^s D_su_\eps(r)\delta W(r)\Big|^{n/2}\Big)^{2/n}ds \Bigg)^{n/2}.
	\end{aligned}
	\end{equation}
	On the other hand, one has from H\"older's inequality and Proposition \ref{prop-meyer2} that for $n>2$
	\begin{equation*}
	\begin{aligned}
	\E\Big|&u_\eps(s)\int_0^s D_su_\eps(r)\delta W(r)\Big|^{n/2}\\
	\leq& \Big(\E|u_\eps(s)|^n\Big)^{1/2}	\Big(\E\Big|\int_0^s D_su_\eps(r)\delta W(r)\Big|^n\Big)^{1/2}\\
		\leq&\Big(\E|u_\eps(s)|^n\Big)^{1/2} K_n^{n/2}\bigg(\Big(\E \|D_su_\eps\|_{L^2([0,1])}^n\Big)^{ 1/n}+\Big(\E\|DD_su_\eps\|_{L^2((0,1)^2)}^n\Big)^{1/n}\bigg)^{n/2}\\
	\leq&K_n^{n/2}\Big(\E|u_\eps(s)|^n\Big)^{1/2} \Bigg(\bigg(\E\Big( \int_0^1|D_su_\eps(r)|^2dr\Big)^{n/2}\bigg)^{ 1/n}+\bigg(\E\Big(\int_0^1\int_0^1 | D_tD_su_\eps(r)|^2drdt\Big)^{n/2}\bigg)^{ 1/n}\Bigg)^{n/2}\\
	\leq&K_n^{n/2}\Big(\E|u_\eps(s)|^n\Big)^{1/2} \Bigg(\bigg(\E \int_0^1|D_su_\eps(r)|^ndr\bigg)^{1/n}+\bigg(\E\int_0^1\int_0^1 | D_tD_su_\eps(r)|^ndrdt\bigg)^{ 1/n}\Bigg)^{n/2}\\
	\leq&2^{n/2}K_n^{n/2}\Big(\E|u_\eps(s)|^n\Big)^{1/2}\bigg(\Big( \int_0^1\E|D_su_\eps(r)|^ndr\Big)^{1/2}+\Big(\int_0^1\int_0^1 \E| D_tD_su_\eps(r)|^{n}drdt\Big)^{1/2}\bigg),
	\end{aligned}
	\end{equation*}
	which implies that 
	\begin{equation}\label{eq-prof3-4}
	\begin{aligned}
	\Bigg(\E\Big|u_\eps(s)&\int_0^s D_su_\eps(r)\delta W(r)\Big|^{n/2}\Bigg)^{2/n}\\
	\leq&2K_n\Big(\E|u_\eps(s)|^n\Big)^{ 1/n}\bigg(\Big(\int_0^1\E |D_su_\eps(r)|^ndr\Big)^{1/n}+\Big(\int_0^1\int_0^1 \E| D_tD_su_\eps(r)|^{n}drdt\Big)^{1/n}\bigg),
	\end{aligned}
	\end{equation}
	Combining \eqref{eq-prof3-4}, H\"older's inequality, Proposition \ref{prop-meyer3} and \eqref{eq-asp-3}, we have for $n>2$
	\begin{equation}\label{eq-prof3-5}
	\begin{aligned}
	\int_0^t&\Bigg(\E\Big(u_\eps(s)\int_0^s D_su_\eps(r)\delta W(r)\Big)^{n/2}\Bigg)^{2/n}ds\\
	\leq& 4K_n\bigg(\int_0^1\Big(\E|u_\eps(s)|^n\Big)^{2/n}ds\bigg)^{1/2}\Bigg(\bigg( \int_0^1\int_0^1\E|D_su_\eps(r)|^ndrds\bigg)^{ 1/n}+\bigg(\int_0^1\int_0^1\int_0^1 \E| D_tD_su_\eps(r)|^{n}drdtds\bigg)^{1/n}\Bigg)\\
	\leq &c_6n^{5+\bar\kappa_2+\bar\kappa_4}\eps^{\bar\kappa_1+\bar\kappa_3},
	\end{aligned}
	\end{equation}
	for some constant $c_6$, independent of $n$, $\eps$.
	Combining \eqref{eq-prof3-3} and \eqref{eq-prof3-5}, we have
	\begin{equation}\label{eq-314}
	\begin{aligned}
	\E|Z(1)|^n\leq& c_7^{n/2}\Bigg(n^{1+2\bar\kappa_4}\eps^{2\bar\kappa_3}+n^{6+\bar\kappa_2+\bar\kappa_4}\eps^{\bar\kappa_1+\bar\kappa_3} \Bigg)^{n/2}\\
	\leq& c_8^{n/2}n^{\hat\kappa_2n}\eps^{\hat\kappa_1n},
	\end{aligned}
	\end{equation}
		for some constants $c_7$, $c_8$, independent of $n$, $\eps$.
		
		Now, denote
	$
	\alpha=(0.5+\hat\kappa_1)\beta
	$
	for some $\beta$ satisfying
	$
	\beta\hat\kappa_2<1.
	$
	We have from \eqref{eq-314} that
	$$
	\begin{aligned}
	\E \exp\{|\eps^{-\hat\kappa_1}\delta(u_\eps)|^{\beta}\}&=\E \exp\{|\eps^{-\hat\kappa_1}Z(1)|^{\beta}\}=\sum_{n=0}^\infty \frac{\E |\eps^{-\hat\kappa_1}Z(1)|^{n\beta}}{n!}\\
	&\leq c_8+\sum_{n=2}^\infty \frac{c_9^{n\beta/2}n^{\hat\kappa_2n\beta}}{n!}
	\leq c_{10},
	\end{aligned}
	$$
	for some constants $c_8,c_9,c_{10}$, independent of $n,\eps$.
	In the above, the last estimate follows from the fact that
	$\sum_{n=2}^\infty \frac{c_9^{n\beta/2}n^{\hat\kappa_2n\beta}}{n!}<\infty,$
	which is implied
by the fact $\hat\kappa_2\beta<1$ and the ratio test.
	As in the proof of Theorem \ref{thm-2},
	by Markov's inequality one has that for any $L>0$
	$$
	\begin{aligned}
	\limsup_{\eps\to0}\eps^\alpha\log \PP(|F_\eps|>L)
	\leq-L^\beta.
	\end{aligned}
	$$
	Therefore, the exponential tightness with the speed $v(\eps)=\eps^\alpha$ follows immediately.
\end{proof}
\begin{rem}
	It is seen from Theorem \ref{thm-2} and \ref{thm-3} that 
	when $\kappa_1,\hat\kappa_1$ are small and $\bar\kappa_4$ is not too large,
	the exponential tightness in Theorem \ref{thm-3} is stronger than that in Theorem \ref{thm-2}.
\end{rem}

\begin{rem}
	One
can reduce the moment
needed in the conditions \eqref{eq-asp-2} and \eqref{eq-asp-3}. For example, we can replace the term $\|u\|_{L^p(\Omega,H)}=\Big(\E\|u_\eps\|_H^p\Big)^{1/p}$ by a smaller term $\|\E|u_\eps|\|_H$ in \eqref{eq-asp-2} and \eqref{eq-asp-3} by using the argument as in, for example, \cite[Proposition 1.5.8]{Nua06} basing the use of operator $(I-L)^{\frac 12}$ and bounded the operator $R$ (in $L^p(\Omega)$), which is used to exchange the derivative operator. However, we need to pay the cost for that because one may need higher order term than $p^{2}$ to bound the constant $K_p$ in \eqref{eq-prop-1}.
\end{rem}

\section{An Application}\label{sec:app}
This section is devoted to 
an application of our main results.
Let $\{\xi_\eps(t)\}_{\eps>0}$ be
a family of stochastic processes depending on a Brownian motion $W(t)$. We are concerned with the exponential tightness of the following family of random variables
\begin{equation}\label{eq-exp-1}
F_\eps=\sqrt\eps e^{-\frac1{\eps^2}\int_0^1\lambda(\xi_\eps(r))dr}\int_0^1e^{\frac 1{\eps^2}\int_0^s\lambda(\xi_\eps(r))dr}g(\xi_\eps(s))dW(s),
\end{equation}
where $\lambda,g$ are smooth functions, bounded together with their derivatives. Moreover, $\lambda(x)\geq \kappa_0>0, \forall x$.
In many problems in mathematical physics such as Langevin equations, stochastic acceleration, we need to deal with this family and establish its tightness (to obtain the limit behavior, the large deviations principle, the averaging principle, etc); see e.g., \cite{CLL,CF,NY20} and references therein. In general, such a term is often related to the solution of a second-order stochastic differential equations in random environment or in the setting of fast-slow second-order system; see e.g., \cite{NY1}. To be self-contained, we write down a simple Langevin equation with strong damping after scaling the time (see e.g., \cite{CF}) in random environment as follows:
$$
\eps^2\ddot{x}(t)=f(x(t),\xi_\eps(t))-\lambda(\xi_\eps(t))\dot{x}(t)+\sqrt\eps g(\xi_\eps(t))\dot {W}(t).
$$
 By using the variation of parameter formula (see e.g., \cite{CF}), we can obtain explicitly the diffusion part of $x_\eps$. Dealing with this part requires the treatment of the family $\{F_\eps\}_{\eps>0}$ defined as in \eqref{eq-exp-1}; see e.g., \cite{CF,NY20,NY1}.
The non-adaptedness of $e^{-\frac 1{\eps^2}\int_0^1 \lambda(\xi_\eps(r))dr}$ is a main challenge here because we cannot move it
inside the stochastic integral in It\^o's sense and estimates for martingales are no longer valid. Meanwhile, we really need such variable to balance the large factor $e^{\frac 1{\eps^2}\int_0^s\lambda(\xi_\eps(r))dr}$ inside the stochastic integral. In the literature, much effort has been devoted to overcoming this challenge. For example, if we consider the case where $\xi_\eps(t)$ is continuously differentiable (or piecewise continuously differentiable), Cerrai and Freidlin \cite{CF} have tried to interpret $\int_0^1e^{\frac 1{\eps^2}\int_0^s\lambda(\xi_\eps(r))dr}g(\xi_\eps(s))dW(s)$ in the pathwise sense. But this approach is no longer valid without the regularity of the random environment and also we cannot cancel out effectively the large factor $\frac 1{\eps^2}$ to provide estimates in probability; see the details in \cite{CLL,CF,NY20}. Another approach in \cite{NY1} is to decompose $\lambda(\xi_\eps(s))$ into two parts, one of them is adapted and the other is ``controllable". However, this approach requires the decay of the derivative of $\lambda$ to control such a decomposition.

With the results developed in this work, we propose a new approach for such problems.
Using 
Malliavin calculus and integration
by parts (see \cite[Proposition 1.3.3]{Nua06}), we can write
$$
\begin{aligned}
F_\eps=&\sqrt\eps \int_0^1 e^{-\frac 1{\eps^2}\int_s^1\lambda(\xi_\eps(r))dr}g(\xi_\eps(s))\delta W(s)\\
&+\sqrt\eps \int_0^1 e^{\frac 1{\eps^2}\int_0^s\lambda(\xi_\eps(r))dr}g(\xi_\eps(s))D_se^{-\frac1{\eps^2}\int_0^1\lambda(\xi_\eps(r))dr}ds\\
=:&F_\eps^{(1)}+F_\eps^{(2)}.
\end{aligned}
$$
In fact,
if  $\{F_\eps^{(1)}\}_{\eps>0}$ is exponentially tight with the speed $v_1(\eps)$ and  $\{F_\eps^{(2)}\}_{\eps>0}$ is exponentially tight with
the speed $v_2(\eps)$, we will obtain that the family $\{F_\eps\}_{\eps>0}$ is exponentially tight with the speed $v(\eps)=\max\{v_1(\eps),v_2(\eps)\}$.
In the below, we use
the letter $c$ to represent universal constants (independent of $\eps,p$), whose values may
change for different usage.

Because there is no stochastic integral involving $F_\eps^{(2)}$, the family $\{F_\eps^{(2)}\}$ can be handled in the usual methodology in the literature.
Indeed,
we have that
$$
\begin{aligned}
D_se^{-\frac1{\eps^2}\int_0^1\lambda(\xi_\eps(r))dr}
=&-\frac 1{\eps^2}e^{-\frac1{\eps^2}\int_0^1\lambda(\xi_\eps(r))dr}\int_0^1\lambda'(\xi(r))D_s\xi_\eps(r)dr.
\end{aligned}
$$
So, one gets
$$
F_\eps^{(2)}=\eps^{-3/2} \int_0^1e^{-\frac 1{\eps^2}\int_s^1\lambda(\xi_\eps(r))dr}g(\xi_\eps(s))\int_0^1\lambda'(\xi_\eps(r))D_s\xi_\eps(r)drds,
$$
and thus,
$$
|F_\eps^{(2)}|\leq c\eps^{-3/2} \int_0^1 e^{-\frac {\kappa_0(1-s)}{\eps^2}}\|D_s\xi_\eps\|_{L^1([0,1])}ds.
$$
For any $p\in(1,\infty]$, H\"older's inequality yields that
$$
\begin{aligned}
|F_\eps^{(2)}|\leq& c\eps^{-3/2}\Big(\int_0^1  e^{-\frac {p\kappa_0(1-s)}{(p-1)\eps^2}}ds\Big)^{\frac{p-1}{p}}\Big(\int_0^1\|D_s\xi_\eps\|_{L^1([0,1])}^pds\Big)^{\frac 1p}\\
\leq &c\eps^{\frac {p-4}{2p}}\Big(\int_0^1\|D_s\xi_\eps\|_{L^1([0,1])}^pds\Big)^{\frac 1p}.
\end{aligned}
$$
Therefore, establishing the exponential tightness for $\{F_\eps^{(2)}\}_{\eps>0}$ reduces to establishing the exponential tightness for
$\Big(\int_0^1\|D_s\xi_\eps\|_{L^1([0,1])}^pds\Big)^{1/p}$. 
Thus, under certain conditions on $\xi_\eps(s)$, we can obtain the exponential tightness of $\{F_\eps^{(2)}\}_{\eps>0}$.

The challenge, which we now should focus more on, is to handle the family $\{F_\eps^{(1)}\}_{\eps>0}$, which is in fact a family of Skorohod integrals. 
By applying our results, we can obtain the exponential tightness of $\{F_\eps^{(1)}\}_{\eps>0}$ under certain conditions.
First, it is readily seen that
$$
\int_0^1e^{-\frac 1{\eps^2}\int_s^1\lambda(\xi_\eps(r))dr}g(\xi_\eps(s))ds\leq c\eps^2.
$$
On the other hand, we have
\begin{equation}\label{eq-514}
\begin{aligned}
D_t&e^{-\frac1{\eps^2}\int_s^1\lambda(\xi_\eps(r))dr}g(\xi_\eps(s))\\
=&-\frac 1{\eps^2}e^{-\frac1{\eps^2}\int_s^1\lambda(\xi_\eps(r))dr}g(\xi_\eps(s))D_t\int_s^1\lambda(\xi_\eps(r))dr+e^{-\frac1{\eps^2}\int_s^1\lambda(\xi_\eps(r))dr}g'(\xi_\eps(s))D_t\xi_\eps(s)
\\
=&-\eps^{-2}e^{-\frac1{\eps^2}\int_s^1\lambda(\xi_\eps(r))dr}g(\xi_\eps(s))\int_s^1\lambda'(\xi_\eps(r))D_t\xi_\eps(r)dr+e^{-\frac1{\eps^2}\int_s^1\lambda(\xi_\eps(r))dr}g'(\xi_\eps(s))D_t\xi_\eps(s).
\end{aligned}
\end{equation}
Therefore, by direct computations, one has
$$
|D_te^{-\frac1{\eps^2}\int_s^1\lambda(\xi_\eps(r))dr}g(\xi_\eps(s))|^2\leq c\eps^{-4}e^{-\frac {2\kappa_0(1-s)}{\eps^2}}\int_0^1 |D_t\xi_\eps(r)|^2dr+c|D_t\xi_\eps(s)|^2,
$$
and 
thus
$$
\begin{aligned}
\int_0^1\int_0^1|D_te^{-\frac1{\eps^2}\int_s^1\lambda(\xi_\eps(r))dr}g(\xi_\eps(s))|^2dtds\leq c\eps^{-2}\|D\xi_\eps\|_{L^2([0,1]^2)}^2.
\end{aligned}
$$
Therefore, by applying Theorem \ref{thm-2},
under certain condition on $\|D\xi_\eps\|_{L^2(\Omega,L^2([0,1]^2))}$,
we can obtain the exponential tightness for $\{F_\eps^{(1)}\}_{\eps>0}$ with some speed $v_1(\eps)$.
To be clear, let us state an explicit result as the following theorem, which follows immediately from Theorem \ref{thm-2}.

\begin{thm}\label{thm-41}
	Assume that the family of random environments $\xi_\eps$ is such that there are constants $\kappa_1>1/2$, $\kappa_2\geq 0$ satisfying
	$$
	\|D\xi_\eps\|_{L^{p}(\Omega,L^2([0,1]^2))}\leq c\eps^{\kappa_1}p^{\kappa_2},\;\forall \eps>0, p\geq 1,
	$$
	for some universal constant $c$.
	The family $\{F_\eps^{(1)}\}_{\eps>0}$ 
	is exponentially tight with the speed $v_1(\eps)=\eps^\alpha$ for any $\alpha$ satisfying
	$$
	\alpha<\frac{\kappa_1-0.5}{5+\kappa_2}.
	$$
\end{thm}

One may worries about the condition $\kappa_1>1/2$ in Theorem \ref{thm-41}. In particular, if $\xi_\eps$ is independent of $\eps$, such a ``decaying condition" on $\eps$ may be violated.
In that case, we can modify the above as follows. We have from \eqref{eq-514} that
\begin{equation}\label{eq-514-2}
|D_te^{-\frac1{\eps^2}\int_s^1\lambda(\xi_\eps(r))dr}g(\xi_\eps(s))|^2\leq c\eps^{-4}e^{-\frac {2\kappa_0(1-s)}{\eps^2}}(1-s)\sup_{r\in[0,1]}|D_t\xi_\eps(r)|^2+c|D_t\xi_\eps(s)|^2.
\end{equation}
A change of variable  leads to
\begin{equation}\label{eq-B2-2}
\int_0^1 \exp\left\{\frac{-\kappa_0 s}{\eps^2}\right\}\cdot\frac{s}{\eps^2}ds=\eps^2\int_0^{\frac 1{\eps^2}}e^{-\kappa_0 r}rdr
\leq c\eps^2.
\end{equation}
Combining \eqref{eq-514-2} and \eqref{eq-B2-2}, one gets that
$$
\begin{aligned}
\int_0^1\int_0^1|D_te^{-\frac1{\eps^2}\int_s^1\lambda(\xi_\eps(r))dr}g(\xi_\eps(s))|^2dtds\leq c\int_0^1\sup_{r\in[0,1]}|D_t\xi_\eps(r)|^2dt.
\end{aligned}
$$
\begin{thm}\label{thm-42}
	Assume that the family of random environments $\xi_\eps$ is such that there are constants $\kappa_1>-1/2$, $\kappa_2\geq 0$ satisfying
	$$
	\Big\|\int_0^1\sup_{r\in[0,1]}|D_t\xi_\eps(r)|^2dt\Big\|_{p}\leq c\eps^{\kappa_1}p^{\kappa_2},\;\forall \eps>0, p\geq 1,
	$$
	for some universal constant $c$.
	The family $\{F_\eps^{(1)}\}_{\eps>0}$ 
	is exponentially tight with the speed $v_1(\eps)=\eps^\alpha$ for any $\alpha$ satisfying
	$$
	\alpha<\frac{\kappa_1+0.5}{5+\kappa_2}.
	$$
\end{thm}

\

\noindent{\bf Acknowledgment}. The author 
thanks his supervisor Professor George Yin for
reading 
an early version of the manuscript and his tireless guidance and support;
he also thanks
Professor David Nualart for the
discussion on Meyer's inequality.

\end{document}